\documentclass[11pt,reqno]{amsart}
\usepackage[T1]{fontenc}
\usepackage{graphicx}

\usepackage{color}
\definecolor{MyLinkColor}{rgb}{0,0,0.4}

\newcommand{\R}{{\mathbb R}}
\newcommand{\E}{{\mathcal E}}

\newcommand{\N}{{\mathbb N}}

\newcommand{\B}{\mathcal{B}}

\newcommand{\cK}{\mathcal{K}}

\newcommand{\p}{\partial}

\newcommand{\e}{\varepsilon}

\newtheorem{thm}{Theorem}[section]
\newtheorem{prop}[thm]{Proposition}
\newtheorem{lemma}[thm]{Lemma}

\newtheorem{defn}[thm]{Definition}
\theoremstyle{remark} 
\newtheorem{rem}[thm]{Remark}
\numberwithin{equation}{section}

\title[Finite speed of propagation and waiting time for a thin film Muskat problem]{Finite speed of propagation and waiting time for\\ a thin film Muskat problem}
\thanks{Partially supported by the French-German PROCOPE project 30718ZG}

\author[Ph. Lauren\c cot]{Philippe Lauren\c cot}
\address{Institut  de Math\'ematiques de Toulouse,  UMR 5219, Universit\'e de Toulouse, CNRS, F-31062  Toulouse cedex 9, France}
\email{laurenco@math.univ-toulouse.fr}

\author[B.--V. Matioc]{Bogdan--Vasile Matioc}
\address{Institut f{\"u}r Angewandte Mathematik, Leibniz Universit{\"a}t Hannover, Welfengarten~1, 30167 Hannover, Germany.}
\email{matioc@ifam.uni-hannover.de}


\usepackage[colorlinks=true,linkcolor=MyLinkColor,citecolor=MyLinkColor]{hyperref} 

\begin{document}

\begin{abstract}
Finite speed  of propagation is established for  non-negative weak solutions to a thin film approximation of the two-phase Muskat problem. The temporal expansion rate of the support matches the scale invariance of the system.
Moreover, we determine sufficient conditions on the initial data for the occurrence of
waiting time phenomena. 
\end{abstract}

\maketitle

\section{Introduction and main results}\label{Sec1}

The  Muskat problem is a  complex free boundary model which was  proposed by Muskat \cite{Mu34} to describe the motion of two immiscible  fluids with different densities and viscosities in a porous medium with impermeable bottom 
(such as intrusion of water into oil).
In the limit of thin fluid layers it was shown in \cite{EMM12} that the   Muskat problem can be approximated by a strongly coupled parabolic system of equations which, when neglecting   surface tension effects,   reads as follows    
\begin{subequations} \label{eq:problem}
\begin{equation}\label{eq:S2}
\left\{
\begin{array}{rcl}
\p_t f & = & \p_x\left( f \p_x\left(  (1+R) f +R g\right)\right),\\[1ex]
\p_t g & = & R_\mu\p_x\left( g \p_x\left(  f + g \right)\right),
\end{array}
\right.
\qquad (t,x)\in (0,\infty)\times \R,
\end{equation}
and is  supplemented with   initial conditions
\begin{equation}\label{eq:bc1}
f(0)=f_0,\qquad g(0)=g_0, \quad x\in \R.
\end{equation}
\end{subequations}
The  constants $R$ and $R_\mu$  in \eqref{eq:S2}, which are assumed in this paper to be positive,  are defined as
\[
R:=\frac{\rho_+}{\rho_--\rho_+}\quad\text{and}\quad R_\mu:=\frac{\mu_-}{\mu_+}R,
\]
with $\rho_-$ and $\mu_-$ [resp. $\rho_+$ and $\mu_+$] denoting the density and viscosity of the lower fluid  
[resp. of the upper fluid].
This reduced model   retains only  the functions $f=f(t,x)\ge 0$ and $g=g(t,x)\ge 0$ as unknowns, where 
  $f $ is the thickness of the lower fluid layer and  $g $ is the thickness of the upper fluid layer, so that $f+g$ 
is the total height of the fluids.
When $R_\mu=R$ the system \eqref{eq:S2} is also a particular case of thin film models derived in  \cite{JM14}  in the context of seawater intrusion.

The system \eqref{eq:S2} is a degenerate parabolic system with a full diffusion matrix and it can be regarded as a
two-phase generalization of the porous medium equation. 
Among salient features of the latter are the finite speed of propagation and   waiting time phenomena. 
Recall that the former means that   the  support of solutions remains compact if it is initially compact, while a waiting time phenomenon refers to the situation   where the solution vanishes at a point of the boundary of the support of its initial condition for   some time.
Since the system \eqref{eq:S2} is degenerate and somewhat related to the porous medium equation,
  these two issues  are  questions which arise   naturally and the purpose of this paper is to provide an affirmative answer to both.

There is a huge  literature on  the finite speed of propagation for degenerate parabolic  equations and various methods 
have been developed to investigate this issue.
  In particular, for second order parabolic equations, such as the porous medium equation or the $p$-Laplacian equation,
for which the comparison principle is available, this property can be derived by comparison with suitable subsolutions and supersolutions, see \cite{Va07} and the references therein. 
This approach  however cannot be extended  to higher order equations or to systems,   
  and energy methods have been developed instead, see  \cite{ ADS02, Be96, DPGG03, Sh92} and the references therein.
 These methods were applied in particular to the thin film equation which is a fourth order degenerate parabolic equation
 and also work  for second order equations.
 A few applications to systems of equations can be found in the literature: finite speed of propagation and the occurrence of waiting time phenomena are shown in \cite{DGJ99} for the Poisson-Nernst-Planck system which is of diagonal type with lower
order coupling and in \cite{Fi13} for the parabolic-elliptic chemotaxis Keller-Segel system which one can view as a
nonlocal parabolic equation.  

As we shall see below the energy method is sufficiently flexible to be adapted to study   the strongly coupled degenerate parabolic system \eqref{eq:S2}.  
  Before stating our result let us introduce the notion of weak solution to   \eqref{eq:problem} to be used 
  hereafter.
Let $\cK$ denote  the 
positive cone  of the Banach space $L_1(\R, (1+x^2) dx) \cap L_2(\R)$ defined by 
\begin{equation}
\cK:=\left\{u\in L^1(\R, (1+x^2) dx) \cap L^2(\R)\,:\, u\ge 0 \right\}\!, \label{eq:sup2}
\end{equation}
and set $  \cK^2:= \cK\times\cK.$ 

\begin{defn}\label{Def1} 
Given    $(f_0,g_0)\in\cK^2,$  a pair    $(f,g):[0,\infty)\to  \cK^2$ is a weak solution to \eqref{eq:problem} if
\begin{itemize}
\item[$(i)$] $(f,g)\in L_\infty(0,\infty; L_2(\R;\R^2))$, $ (f,g) \in L_2(0,\infty;H^1(\R;\R^2))$,\\[-2ex]
\item[$(ii)$] $(f,g)\in C ([0,\infty);H^{-1}(\R;\R^2))$ with $(f,g)(0)=(f_0,g_0),$
\end{itemize}
and    $(f,g)$ solves the equations \eqref{eq:S2} in the following sense
\begin{equation}\label{T2}   
\left\{
\begin{aligned}
&\int_\R f(t)\ \xi\, dx-\int_\R f_0\ \xi\, dx + \int_0^t\int_\R f(\sigma) \left[ (1+R)\p_xf+R\p_xg\right](\sigma) \p_x\xi\, dx\, d\sigma=0,\\[1ex]
&\int_\R g(t)\ \xi\, dx - \int_\R g_0\ \xi\, dx+ R_\mu\ \int_0^t \int_\R g(\sigma) \left( \p_xf+ \p_xg \right)(\sigma)\p_x\xi\, dx\, d\sigma=0
\end{aligned}\right.
\end{equation}
for all $\xi\in C_0^\infty(\R) $ and $t\ge 0.$
\end{defn}
The existence of   weak solutions to \eqref{eq:problem}   is shown
in \cite{LM13}   by  
  a variational scheme. The proof   relies on the observation 
that the system \eqref{eq:S2} is a gradient   flow with respect to the $2$-Wasserstein metric of the following
energy functional 
\begin{equation}\label{eq:sup1}
\E(f,g):=\frac{1}{2}\int_\R \left[ f^2+R(f+g)^2 \right]\, dx.
\end{equation} 
This approach actually extends to the  two dimensional setting as well as to 
 a related fourth order degenerate  system  which is also a thin film approximation of the Muskat problem additionally incorporating surface tension effects   \cite{LM14}.
Let us point out that the uniqueness  of solutions to  \eqref{eq:problem} is an open problem.
 \pagebreak

The main results of this paper are the following.
\begin{thm}[Finite speed of propagation]\label{MT1} Let $(f,g)$ be a  weak  solution of \eqref{eq:problem}.
If   $(f,g)$ satisfies the local energy estimate
\begin{align} \label{WL2}
& \int_\R\big[ f ^2(T) + R (f+g)^2(T)\big]\zeta^2\, dx\nonumber\\
& \qquad + \int_0^T\int_\R \big(f \left| (1+R)\p_x f + R \p_xg \right|^2+ RR_\mu g  \left| \p_x f + \p_x g \right|^2\big)\zeta^2\, dx\, dt\nonumber \\ 
& \leq \int_\R\big[ f^2 (0) + R (f+g)^2(0)\big]\zeta^2\, dx \nonumber \\
& \qquad +4\int_0^T\int_\R \big[f \left( (1+R)f+Rg \right)^2+RR_\mu  g \left( f + g \right)^2\big] |\p_x\zeta|^2\, dx\, dt
\end{align}
for all $\zeta\in W^1_4(\R)$ as well as for $\zeta\equiv 1$, then  $(f,g)$ has finite speed of propagation.
More precisely, if $a\geq0, r_0>0,$  and ${\rm supp\, } (f_0+g_0)\,\cap\,(a-r_0,a+r_0)=\emptyset$, 
then there exists a positive constant $C_*=C_*(R,R_\mu)$ such that  
\[{\rm supp \,} (f(T)+g(T))\,\cap\, (a-r_0/2,a+r_0/2)=\emptyset
\qquad\text{for all $T\in\big(0, C_* r_0^{5/2}/ \E^{1/2}(f_{0},g_{0})\big]$}.\]
In particular, if ${\rm supp\,} (f_0+g_0)\subset [-b_0,b_0]$, with $b_0>0$,  then there exists a positive constant $C^*=C^*(R_\mu,R,f_0,g_0)$ such that  
\[{\rm supp \,} (f(T)+g(T))\subset\big[ -b_0 -C^* T^{1/3} ,b_0+C^* T^{1/3}  \big]\qquad\text{for all $T>0$.}\]
\end{thm}

We note that Theorem \ref{MT1} is only valid for weak solutions which satisfy  in addition the local energy estimate \eqref{WL2}.
Unfortunately, we are yet unable  to derive it for arbitrary weak solutions and it is in particular unclear whether
 it holds true for the weak solutions we constructed in \cite{LM13}. 
We shall show in Section \ref{Sec3} that for each initial data there is at least a weak solution to \eqref{eq:problem}
satisfying the local estimate \eqref{WL2}.
To this end we will adapt an approximation scheme from \cite{ELM11} which allows us to obtain a weak solution as a limit of classical solutions to a regularized version of the original system.

Let us also mention that Theorem \ref{MT1} gives no clue concerning the finite speed of propagation for each component taken
separately. 

\begin{rem}\label{R1}
\begin{itemize}
\item[$(a)$] It is shown in \cite{LM15} that the system \eqref{eq:S2} has  self-similar solutions  of the type
\[
\big[(t,x)\mapsto  (1+t)^{-1/3}(F,G)( (1+t)^{-1/3}x)\big],\qquad\text{for $t\geq0$ and $x\in\R$,}
\]
with compactly supported profiles $(F,G)\in H^1(\R,\R^2)\cap\cK^2.$ Hence, the estimate on the growth
rate of the support   obtained in Theorem \ref{MT1} matches that  of the self-similar solutions and is likely to be optimal.

\item[$(b)$] The constant $C^*$ in the last statement of Theorem \ref{MT1} only depends on  $f_0$ and $g_0$ through the energy $\E(f_0,g_0)$ and the second moments of
$f_0$ and $g_0$.
\end{itemize}
\end{rem}

Due to \cite{DPGG03}, a direct consequence of the local energy estimate \eqref{WL2} is 
the occurrence of waiting time phenomena.
\begin{thm}[Waiting time phenomena]\label{MT2}
 Let $(f,g)$ be a  weak  solution of \eqref{eq:problem} such that \eqref{WL2} holds for all $\zeta\in W^1_4(\R)$. Let  
$x_0\in\overline{\R\setminus{\rm supp\, } (f_0+g_0)}$ be such that 
$$
\limsup_{r\to 0}\frac{1}{r^5}\int_{x_0-r}^{x_0+r} \left[ f^2_0+R(f_0+g_0)^2 \right]\, dx<\infty.
$$
Then there exists a positive time $T_*$ such that $x_0\in\overline{\R\setminus{\rm supp\, } (f(T)+g(T))}$ for all $T\in (0,T_*)$.
\end{thm}

Let us now describe the content of this paper:  Section \ref{Sec2} is devoted
to the proof of the main results. While  Theorem \ref{MT2} is a straightforward consequence   
of \eqref{WL2} and \cite[Theorem 1.2]{DPGG03}, the proof of  Theorem \ref{MT1} requires several steps and is inspired from 
\cite{Be96} which deals with the thin film equation. It is worth pointing out that fewer estimates are available for the system \eqref{eq:problem} as in \cite{Be96}.  The last section is devoted to the existence of weak solutions to \eqref{eq:problem} satisfying the local energy estimate \eqref{WL2}.


\section{Finite speed of propagation}\label{Sec2}

Throughout this section, $(f,g)$ is a  weak  solution of \eqref{eq:problem} which satisfies the local energy estimate 
\eqref{WL2} and 
\begin{equation}
w:=[f^2+R(f+g)^2]^{3/4}. \label{brahms}
\end{equation}
The function $w$ inherits some regularity properties of $(f,g)$ as shown in the following result.

\begin{lemma}\label{LA1}
Given non-negative functions $u,v\in H^1(\R)$, let 
\[
z:=(u^2+Rv^2)^{3/4}.
\]
Then $z\in H^1(\R)$ and 
\[\p_xz=\frac{3}{2}\frac{u\p_xu+Rv\p_xv}{{\bf 1}_{\{0\}}(z)+(u^2+Rv^2)^{1/4}},\]
where $ {\bf 1}_E$ is the characteristic function of the set $E.$
\end{lemma}

\begin{proof} We choose positive functions 
$u_n,v_n\in C^\infty(\R)\cap H^1(\R), n\geq1$, such that $u_n\to u$ and $v_n\to v$ in $H^1(\R)$ and set
 \[
z_n:=(u_n^2+Rv^2_n)^{3/4}.
\]
Obviously $z_n^{4/3}\to z^{4/3}$ in $L_1(\R)$ and it follows from the H\"older continuity of the function $[x\mapsto |x|^{3/4}]$ that 
\begin{align*}
|z_n-z|^{4/3}\leq &|z_n^{4/3}-z^{4/3}|\qquad\text{for all $n\in\N$,}
\end{align*}
hence $z_n\to z$ in $L_{4/3}(\R).$ We next note that the sequence $(z_n)_n$ is bounded in $ H^1(\R) $ so that it has a subsequence which converges weakly in $H^1(\R)$ towards a limit  which coincides with $z$ almost everywhere. Consequently $z$ belongs to $H^1(\R)$ and the formula for $\p_x z$ follows by standard arguments.
\end{proof}

We now derive from  \eqref{WL2}  a local energy estimate for the function $w$ defined in \eqref{brahms} which is at the heart
of our analysis. 
\begin{lemma}\label{LA2} 
The function $w$ defined in \eqref{brahms} satisfies 
\begin{align} \label{WLB}
 \int_\R w^{4/3}(T)\zeta^2\, dx+C_1\int_0^T\int_\R |\p_x w|^2\zeta^2\, dx\, dt
 \leq &  \int_\R w^{4/3}(0)\zeta^2\, dx+C_2\int_0^T\int_{\R} w^2|\p_x\zeta|^2 \, dx\, dt
\end{align}
for all $T>0$ and all $\zeta\in W^1_4(\R)$.
The constants $C_1$ and $C_2$ depend only on $R$ and $R_\mu$.
\end{lemma}

\begin{proof} By Lemma~\ref{LA1} the function $w$ belongs to $ H^1(\R)$ and
\begin{align}
|\p_xw|^2 & = \left|\frac{f[(1+R)\p_xf+R\p_xg]+R g (\p_xf+\p_xg)}{{\bf 1}_{\{0\}}(w)+(f^2+R(f+g)^2)^{1/4}}\right|^2 \nonumber \\
& \le \frac{2f}{{\bf 1}_{\{0\}}(w)+(f^2+R(f+g)^2)^{1/2}}  f| (1+R)\p_xf+R\p_xg |^2 \nonumber \\
& \qquad +\frac{2R^2g}{{\bf 1}_{\{0\}}(w)+(f^2+R(f+g)^2)^{1/2}}  g|  \p_xf+ \p_xg |^2 \nonumber \\
& \le 2\max\left\{1,\frac{\sqrt{R}}{R_\mu}\right\} \left[ f| (1+R)\p_xf+R\p_xg |^2+RR_\mu  g|  \p_xf+ \p_xg |^2\right]. \label{haendel}
\end{align}
In addition, since $w^{4/3} \ge \max{\left\{ (1+R) f^2 , R g^2 \right\}}$, there holds
\begin{align*}
f \left( (1+R)f+Rg \right)^2+RR_\mu  g \left( f + g \right)^2\leq &2\max\left\{\frac{R_\mu}{\sqrt R},
 \sqrt{1+R}\right\}w^2.
\end{align*}
Combining these two inequalities with \eqref{WL2} gives the claim.
\end{proof}

We next recall that the Gagliardo-Nirenberg inequality  \cite[Theorem 1]{N66} states that there exists a constant $C>0$ such that
\[
\|v\|_2\leq C\|\p_x v\|_2^{1/5}\|v\|_{4/3}^{4/5}+C\|v\|_{4/3}\qquad\text{for all $v\in H^1((-1,1))$.}
\]  
Using a scaling argument, we deduce from the inequality above that, for $r>0$,
\begin{equation}\label{GN}
\|v\|_2\leq C\|\p_x v\|_2^{1/5}\|v\|_{4/3}^{4/5}+Cr^{-1/4}\|v\|_{4/3} \qquad\text{for all $v\in H^1((-r,r))$.}
\end{equation}
A consequence of the Gagliardo-Nirenberg inequality \eqref{GN} is the following interpolation inequality in the spirit 
of \cite[Lemma~10.1]{Be96}.

\begin{lemma}\label{LA3} 
There is $C_3>0$ such that, given $r>0$ and $v\in H^1((-r,r))$, there holds
\begin{equation}\label{FGN}
\|v\|_2^2\leq C_3\|\p_xv\|_2^{14/11} I_r^{6/11}+C_3r^{-7/2}I_r^{3/2},  
\end{equation}
where
\[
I_r:=\int_{-r}^r (r-|x|)_+^2 |v(x)|^{4/3}\, dx.
\]
\end{lemma}

\begin{proof} We pick $\rho\in(0,r)$ arbitrary and infer from the H\"older inequality that
\begin{align*}
\int_{-r}^r|v|^{4/3}\, dx & = \int_{-\rho}^\rho|v|^{4/3}\, dx+\int_{\{\rho< |x| < r\}} |v|^{4/3}\, dx \\
& \leq\frac{1}{(r-\rho)^2}\int_{-\rho}^\rho (r-|x|)_+^2 |v|^{4/3}\, dx +2(r-\rho)^{1/3}\|v\|_2^{4/3}.
\end{align*}
We now choose $\rho\in(0,r)$ such that 
\[\frac{1}{(r-\rho)^2}\int_{-\rho}^\rho (r-|x|)_+^2 |v|^{4/3}\, dx=2(r-\rho)^{1/3}\|v\|_2^{4/3}\]
and we obtain
\begin{align}\label{*}
\|v\|_{4/3}\leq&4\|v\|_2^{6/7} I_r^{3/28}.
\end{align}
Using \eqref{GN} and \eqref{*} yields
\begin{align*} 
\|v\|_2^2\leq C\|\partial_x v\|_2^{2/5}\big(\|v\|_2^2\big)^{24/35} I_r^{6/35}+Cr^{-1/2}\big(\|v\|_2^2\big)^{6/7} I_r^{3/14}, 
\end{align*}
and thus
\begin{align*} 
\big(\|v\|_2^2\big)^{11/35}\leq C\|\partial_x v\|_2^{2/5}  I_r^{6/35}+Cr^{-1/2}\big(\|v\|_2^2\big)^{6/35} I_r^{3/14}. 
\end{align*}
By Young's inequality we get
\begin{align*} 
\|v\|_2^2\leq& C\|\p_x v\|_2^{14/11}  I_r^{6/11}+Cr^{-35/22}\big(\|v\|_2^2\big)^{6/11} I_r^{15/22}\\
\leq& C\|\p_x v\|_2^{14/11}  I_r^{6/11}+\frac{1}{2}\|v\|_2^2+Cr^{-7/2}  I_r^{3/2},
\end{align*}
and the proof is complete.
\end{proof}

We now introduce additional notation. For $r>0$ and $T>0$ we set
\begin{align*}
&u_k(r,T):=\int_0^T\int_{-r}^r|\p_x w(t,x)|^2 (r-|x|)_+^k\, dx\, dt\qquad \text{for $k\in\{0,1,2\}$},\\
&I(r,T):=\sup_{t\in (0,T)} \int_{-r}^r w^{4/3}(t,x)(r-|x|)_+^2\, dx,
\end{align*}
where $w$ is defined in \eqref{brahms}. We first derive from \eqref{WLB} an inequality relating $I(r,T)$, $u_0(r,T)$, and $u_2(r,T)$ under suitable constraints on $r$ and $T$.
 
\begin{lemma}\label{LA4}
Consider $r_0>0$ such that  ${\rm supp} (f_0+g_0)\cap (-r_0,r_0)=\emptyset.$ There are positive constants $C_4$ and $C_5$ such that, if $T_0>0$ is such that
\begin{align}\label{COND1}
C_4 T_0 \left( \frac{r_0}{2} \right)^{-7/2} I^{1/2}(r_0,T_0)\leq \frac{1}{2},
\end{align}
then
\begin{align}\label{**}
\frac{1}{3} I(r,T) + u_2(r,T)\leq C_5 T^{4/5} u_0^{7/5}(r,T) 
\end{align}
 for all $r_0/2\leq r\leq r_0$ and $ 0<T\leq T_0$. 
\end{lemma}

\begin{proof} 
Let $T\in (0,T_0]$ and $r\in (0,r_0]$. Setting $\zeta(x) := (r-|x|)_+$, $x\in\R$, we observe that the assumptions on $f_0+g_0$ guarantees that $\zeta^2(x)  w^{4/3}(0,x) = 0$ for $x\in\R$ and we infer from \eqref{WLB} that
\begin{align*} 
 \int_\R w^{4/3}(t)\zeta^2\, dx+C_1\int_0^t\int_\R |\p_x w|^2\zeta^2\, dx\, ds
 \leq &  C_2\int_0^t\int_{-r}^r w^2 \, dx\, ds
\end{align*}
for all $t\in (0,T)$. Hence, there exists a constant $C>0$ such that 
\begin{align*}
I(r,T)+u_2(r,T)\leq C&\int_0^T \int_{-r}^r w^2\, dx ds.
\end{align*}
Using \eqref{FGN} and the H\"older inequality, it follows that 
\begin{align*}
I(r,T)+u_2(r,T)\leq &CTr^{-7/2}I^{3/2}(r,T)+CI^{6/11}(r,T)\int_0^T\|\p_x w(s)\|_{L_2((-r,r))}^{14/11}  \, ds \nonumber\\
 \leq &C_4 Tr^{-7/2}I^{3/2}(r,T)+CT^{4/11}I^{6/11}(r,T)u_0^{7/11}(r,T).
\end{align*}
Since $I(r,T)$ is a nondecreasing  function in both variables $r$ and $T$,  the property \eqref{COND1}, together with Young's inequality and the above inequality, leads us to
\begin{align*}
I(r,T)+u_2(r,T) \leq &\frac{1}{2}I (r,T)+\frac{1}{6}I (r,T)+C_5 T^{4/5} u_0^{7/5}(r,T)
\end{align*}
for all $r_0/2\leq r\leq r_0$ and $ 0<T\leq T_0$. This completes the proof.
\end{proof}

\bigskip
After this preparation, we are in a position to prove our main results.

\begin{proof}[Proof of Theorem~\ref{MT1}]
Since \eqref{eq:problem} is invariant with respect to translations, we may assume that $a=0$ so that ${\rm supp} (f_0+g_0)\cap (-r_0,r_0)=\emptyset.$ Then $w^{4/3}(0,x)=0$ for $x\in (-r_0,r_0)$ and $I(r_0,t)\to 0$ as $t\to 0$,  cf. \eqref{WLB}. Consequently there is $T_0>0$ such that the condition \eqref{COND1} in Lemma~\ref{LA4} is satisfied.

Let $T\in(0,T_0)$. In view of 
$$
\big[x\mapsto \|\p_xw(\cdot, x)\|^2_{L_2(0,T)} \big]\in  L_1(\R),
$$
we have
\begin{equation}\label{DEDE}
\p_r  u_2(r,T)=2u_1(r,T),\qquad \p_r  u_1(r,T)= u_0(r,T) \qquad\text{for  a.e. $r\in(0,r_0),$}
\end{equation}
and the H\"older inequality yields
$$
u_1(r,T)\leq u_2^{1/2}(r,T)u_0^{1/2}(r,T).
$$
This inequality, together with \eqref{**} which is valid here thanks to the choice of $T_0$, gives
$$
u_1(r,T)\leq CT^{2/5} u_0^{6/5}(r,T) = CT^{2/5} (\p_r u_1(r,T))^{6/5} \qquad\text{for  a.e.  $r\in(r_0/2,r_0).$}
$$
Equivalently
\begin{equation}
u_1^{5/6}(r,T)\leq  \frac{1}{6C_6} T^{1/3} \p_ru_1(r,T) \qquad\text{for a.e. $r\in(r_0/2,r_0).$}\label{corelli}
\end{equation}

Taking a smaller value of $T_0$ if necessary, we further assume that 
\begin{equation}
T_0^{1/3} u_1^{1/6}(r_0,T_0) \le C_6 \frac{r_0}{2}. \label{COND2}
\end{equation}
Let $T\in (0, T_0]$ and assume for contradiction that $u_1(r_0/2,T)>0$. Together with the monotonicity properties of $u_1$ this implies that $u_1(r,T)>0$ for all $r\in[r_0/2,r_0]$. Thanks to this positivity property we infer from \eqref{corelli} that
$$
C_6 T^{-1/3}\leq \p_r\left( u_1^{1/6} \right)(r,T) \qquad\text{for  a.e.  $r\in(r_0/2,r_0).$}
$$
After integration we end up with
$$
C_6 T^{-1/3}(r_0-r) \le u_1^{1/6}(r_0,T) - u_1^{1/6}(r,T),
$$
or equivalently
\begin{equation}
T^{1/3} u_1^{1/6}(r,T) \le T^{1/3} u_1^{1/6}(r_0,T) - C_6 (r_0-r), \qquad r\in [r_0/2,r_0]. \label{vivaldi}
\end{equation}
Taking $r=r_0/2$ in \eqref{vivaldi} gives
$$
0 < T_0^{1/3} u_1^{1/6}(r_0,T_0) - C_6 \frac{r_0}{2},
$$
and contradicts \eqref{COND2}. Therefore $u_1(r_0/2,T)=0$ and it follows from \eqref{DEDE} that $u_0(r_0/2,T)=0$ for all 
 $T\in (0,T_0]$. Recalling \eqref{**} we find that 
  $I(r_0/2,T)=0$ for all $T\in (0,T_0]$. 

We further note that, in view of Theorem~\ref{T:1}~$(b)$ and \eqref{haendel},
\begin{align*}
u_1(r,T) & \le r \int_0^T \|\p_x w(t)\|_2^2\, dt \le C r \left[ \mathcal{E}(f_0,g_0) - \mathcal{E}(f(T),g(T)) \right] 
\le C r \mathcal{E}(f_0,g_0), \\
I(r,T) & \le r^2 \sup_{t\in [0,T]} \|w(t)\|_{4/3}^{4/3} \le r^2 \sup_{t\in [0,T]} \mathcal{E}(f(t),g(t)) 
\le r^2 \mathcal{E}(f_0,g_0),
\end{align*}
so that \eqref{COND1} and \eqref{COND2} are satisfied provided $T_0 = C_* r_0^{5/2}/ \E^{1/2}(f_{0},g_{0})$ for a sufficiently small constant $C_*>0$ depending only on $R$ and $R_\mu$. This proves the first claim of Theorem~\ref{MT1}.

Finally, let ${\rm supp\,} (f_0+g_0)\subset [-b_0,b_0]$, with $b_0>0,$ 
and let $T>0$ be fixed. 
Choosing $r_0=(T\E^{1/2}(f_0,g_0)/C_*)^{2/5}$, we have that, for each   $a\geq b_0+ (T\E^{1/2}(f_0,g_0)/C_*)^{2/5}$ 
[resp. $a\leq -b_0- (T\E^{1/2}(f_0,g_0)/C_*)^{2/5}$]
\[{\rm supp\,} (f_0+g_0)\,\cap\, (a-r_0,a+r_0)=\emptyset.\]
We then infer from the first statement of Theorem \ref{MT1} that ${\rm supp\,} (f(T)+g(T))\,\cap\, (a-r_0/2,a+r_0/2)=\emptyset,$
 from which   follows that
 \begin{align}\label{No}{\rm supp \,} (f(T)+g(T))\subset\left[ -b_0-\frac{T^{2/5}\E^{1/5}(f_0,g_0)}{2C_*^{2/5}},
b_0+\frac{ T^{2/5}\E^{1/5}(f_0,g_0)}{2C_*^{2/5}} \right]
\qquad\text{for all $T>0$.}\end{align}
 Consequently,   $(f(T)+g(T))$
is compactly supported for each $T\geq0$ and we set
\[
\beta(T):=\max\big\{b_0,\sup\,({\rm supp\, }(f(T)+g(T)))\big\}.
\]
It then follows that    $\beta(T)\to \beta(0)=b_0$ as $T\to0$.
Since the problem \eqref{eq:S2} is autonomous   the estimate \eqref{No} yields  
\begin{align}\label{BCA}
\beta(T_2)-\beta(T_1)\leq \frac{(T_2-T_1)^{2/5}\E^{1/5}(f(T_1),g(T_1))}{2C_*^{2/5}}\qquad\text{for all $T_2>T_1$.}
\end{align}
Besides, we know from \cite[Theorem 4.1 $(iv)$]{LM15} (after rescaling), that 
\begin{align}\label{BCB}
\E(f(t),g(t))\leq (1+t)^{-1/3}\left[\E(f_0,g_0)+\frac{1}{6}\int_\R\Big(f_0+\frac{R}{R_\mu}g_0\Big)x^2\, dx\right]\leq Ct^{-1/3}\qquad\text{for all $t>0.$}
\end{align} 
Combining \eqref{BCA} and \eqref{BCB} yields
\begin{align*} 
\beta(T_2)-\beta(T_1)\leq C (T_2-T_1)^{2/5}T_1^{-1/15}  \qquad\text{for all $T_2>T_1>0$.}
\end{align*}
We are now in the position to apply \cite[Lemma 7.6]{Be96} to the above functional inequality and conclude that
there exists  a positive constant $C_*0$ depending only on $R,R_\mu, f_0,$ and $g_0$ such that
\[
\beta(T)\leq b_0+C^*T^{1/3}\quad\text{for all $T>0$,}
\]
which is the expected propagation rate. 
The estimate for the expansion of the left boundary of the support is derived in a similar way. 
\end{proof} \medskip

\begin{proof}[Proof of Theorem~\ref{MT2}]
Invoking \eqref{WLB}, Theorem~\ref{MT2} is a particular case of the more general result \cite[Theorem~1.2]{DPGG03} which we apply with $k=1$, $p=2$, and $q=4/3$.
\end{proof}

\section{Weak solutions satisfying the local energy estimate}\label{Sec3}

As already mentioned in the Introduction, we now check that there exists at least a weak solution to \eqref{eq:problem}
satisfying the local energy estimate \eqref{WL2}.  
\begin{thm}[Existence of weak solutions]\label{T:1}  
Given    $(f_0,g_0)\in\cK^2,$ where $\cK^2$ is defined in \eqref{eq:sup2},   there exists at least   a weak solution $(f,g)$
to \eqref{eq:problem}, satisfying the local energy estimate \eqref{WL2}
 as well as the following estimates
\begin{align*}
(a)\quad &   \|f(T)\|_1=\|f_0 \|_1, \, \|g(T)\|_1=\|g_0 \|_1 ,  \\[1ex]
(b) \quad & \mathcal{H}(f(T),g(T)) +\frac{R}{1+2R}\int_0^T\int_\R\big[  |\p_xf|^2 + R |\p_x(f+g)|^2 \big]\, dx\, dt\leq  \mathcal{H}(f_0,g_0),\\[1ex]
(c)\quad &\E(f(T), g(T))
+ \frac{1}{2}\int_{0}^T\int_\R\big[f\left((1+R)\p_xf+R\p_x g\right)^2+RR_\mu g(\p_xf+\p_x g)^2\big]\, dx\,dt \leq \E(f_{0},g_{0})
\end{align*}
 for   all $T\in(0,\infty)$.
  The energy functional  $\E$ is given by \eqref{eq:sup1} and the entropy functional $\mathcal{H}$
  is defined as 
  \begin{equation*} 
\mathcal{H}(f,g) := \int_\R \Big( f \ln{f} +\frac{R}{R_\mu} g \ln{g} \Big)\, dx.
\end{equation*}
\end{thm}\medskip 

The remainder of this section is devoted to the proof of Theorem \ref{T:1}.
We split the proof of Theorem \ref{T:1} in two steps: we first truncate the spatial domain to a finite interval
$(-L,L)$, for some arbitrary  $L>0$, and then introduce a regularized system having global classical solutions.

\subsection{A regularized problem}\label{S3a}

To be more precise, given  $L>0$   and $\e\in(0,1)$  we define  the Hilbert space
\[
H^2_\B:=\{u\in H^2((-L,L))\,:\, \p_x u(\pm L)=0\} 
\] 
and we note that  the elliptic operator  $(1-\e^2\p_x^2):H^2_\B\to L_2((-L,L))$ is an isomorphism.
Setting
\begin{equation}
\mathcal{R}_\e[u]:=(1-\e^2\p_x^2)^{-1} u \in H^2_\B \qquad\text{for $u\in L_2((-L,L))$,}
\label{spirou}
\end{equation}
we consider the following regularized problem
\begin{subequations}\label{eq:RS}
 \begin{equation}\label{eq:A1}
\left\{
\begin{array}{rcl}
\p_t f_\e&=&(1+R)\p_x\left(   f_\e\p_x{f_\e}\right)+R\p_x\left((f_\e-\e)\p_x \mathcal{R}_\e[g_\e]\right),\\[1ex]
\p_tg_\e&=&R_\mu\p_x\left(   (g_\e-\e)\p_x \mathcal{R}_\e[f_\e]\right)+R_\mu \p_x\left(g_\e\p_x{g_\e}\right),
\end{array}
\right.
{ \quad (t,x)\in (0,\infty)\times (-L,L),}
\end{equation}
supplemented with homogeneous Neumann boundary conditions  
\begin{equation}\label{A:2}
\p_xf_\e(t,\pm L) =\p_x g_\e(t,\pm L)=0 , \qquad t\in (0,\infty)  ,
\end{equation}
and with regularized initial data
\begin{equation}\label{eq:A3}
f_\e(0)=f_{0\e}:=\mathcal{R}_\e[f_0{\bf 1}_{(-L,L)}]+\e,\qquad g_\e(0)=g_{0\e}:=\mathcal{R}_\e[g_0{\bf 1}_{(-L,L)}]+\e.
\end{equation}
\end{subequations}
Clearly, the regularized initial data satisfy $(f_{0\e}, g_{0\e})\in H^2_\B\times H^2_\B$ and  
\begin{equation}\label{eq:id0} 
f_{0\e}\geq \e,\qquad  g_{0\e}\geq\e.
\end{equation}

The solvability of problem \eqref{eq:RS} is studied in \cite[Theorem 2.1]{ELM11} with the help of the quasilinear parabolic 
theory developed in \cite{Am93} and we recall the result now. 

\begin{prop}\label{P:3} The problem~\eqref{eq:RS} has a unique non-negative classical solution 
\[
f_\e, g_\e\in C([0,\infty); H^1((-L,L)))\cap C((0,\infty); H^2_\B)\cap  C^1((0,\infty); L_2((-L,L))).
 \]
 Moreover, we have 
 \[
f_\e\geq\e, \qquad g_\e\geq \e\qquad\text{for all }\ (t,x)\in (0,\infty)\times(-L,L),
\]
and
\[\|f_\e(t)\|_1= \|f_{0\e}\|_1=\|f_0{\bf 1}_{(-L,L)}\|_1+2\e L, \quad \|g_\e(t)\|_1=\|g_{0\e}\|_1=\|g_0{\bf 1}_{(-L,L)}\|_1+2\e L\]
for all $t\geq0$.
\end{prop}

The solutions constructed in Proposition \ref{P:3} enjoy additional properties, cf.  \cite[Lemmas 2.4 $\&$ 2.6]{ELM11}. 
 
\begin{lemma}\label{L:3} Given $T\in(0,\infty)$, it holds
 \begin{equation}\label{eq:GE3}
\mathcal{H} (f_\e(T), g_\e(T))+\int_0^T\int_{-L}^L \Big( \frac{1}{2} |\p_xf_\e|^2+ \frac{R}{1+2R} |\p_x g_\e|^2 \Big) \, dx\, dt\leq \mathcal{H} (f_\e(0), g_\e(0))
\end{equation}
and
\begin{eqnarray}
\mathcal{E}_{\e}(f_\e(T),g_\e(T)) & + & \int_0^T \int_{-L}^L\big[ f_\e \left| (1+R) \p_x f_\e + R \p_x G_\e \right|^2 + RR_\mu g_\e \left| \p_x(F_\e+g_\e) \right|^2 \big]\, dx\, dt \nonumber \\
& \le & \mathcal{E}_{\e}(f_{0\e},g_{0\e}) + \e C_2\ \int_0^T\int_{-L}^L \big( |\p_x f_\e |^2 +  |\p_x g_\e |^2\big)\, dx\, dt, \label{gaston}
\end{eqnarray}
with 
\begin{align*}
\mathcal{E}_{\e}(f_\e,g_\e) &:=\frac{1}{2}
\left[(1+R) \|f_\e\|_2^2 + R \|g_\e\|_2^2 + R \int_{-L}^L \left( F_\e g_\e + G_\e f_\e \right)\, dx\right],\\
\mathcal{H}(f_\e,g_\e)&:=\int_{-L}^L \Big( f_\e \ln{f_\e} +\frac{R}{R_\mu}  g_\e \ln{g_\e} \Big)\, dx.
\end{align*}
\end{lemma}

As a consequence of Proposition \ref{P:3} and Lemma \ref{L:3}, the following result is proved in \cite{ELM11}.
\begin{prop}[Weak solutions on a finite interval] \label{P:4} 
There exist a sequence $\e_k\to0$ and a pair   $(f,g) $   
 satisfying
 \begin{itemize}
\item[$(i)$] $f\ge 0$, $g\ge 0$ in $(0,\infty)\times (-L,L),$\\[-2ex]
\item[$(ii)$] $f,g\in L_\infty(0,\infty;L_2((-L,L)))\cap L_2(0,\infty; H^1((-L,L))),$\\[-2ex]
\item[$(iii)$] $f_{\e_k}\to f,$ $ g_{\e_k}\to g$    in $L_2((0,T\times(-L,L)),$
\end{itemize}
 and 
\begin{align}
&\int_{-L}^L f(T)\xi\, dx-\int_{-L}^L f_0\xi\, dx= { -} \int_0^T\int_{-L}^L f\left((1+R)\p_xf+R  \p_xg\right)\p_x\xi\, dx\,dt,\label{FEQ}\\[1ex]
&\int_{-L}^L g(T)\xi\, dx-\int_{-L}^L g_0\xi\, dx={ -} R_\mu\int_0^T\int_{-L}^L g\left(\p_xf+  \p_xg\right)\p_x\xi\, dx\, dt\label{GEQ}
\end{align}
for all $\xi\in W^1_4((-L,L)) $ and all $T>0$. 
Moreover  
\begin{align*}
(a)\quad & \|f(T)\|_1=\|f_0{\bf 1}_{(-L,L)}\|_1, \, \|g(T)\|_1=\|g_0{\bf 1}_{(-L,L)}\|_1, \\[1ex]
(b) \quad &\mathcal{H} (f(T),g(T))+\int_0^T\int_{-L}^L \Big[\frac{1}{2} |\p_xf|^2 + \frac{R}{1+2R} |\p_xg|^2 \Big]\, dx\, dt\leq  \mathcal{H} (f_0,g_0),\\[1ex]
(c)\quad &\E(f(T), g(T))+\int_0^T\int_{-L}^L \big[ f\left((1+R)\p_xf+R\p_xg\right)^2+RR_\mu g(\p_xf+\p_xg)^2 \big]\, dx\, dt\leq \E(f_{0},g_{0})
\end{align*}
for   all  $T\in(0,\infty)$.
\end{prop}


\subsection{A local energy estimate}

We now derive a local version of inequality $(c)$ in Proposition \ref{P:4}.
\begin{lemma}\label{L:4} Let $(f,g)$ be the limit  of $((f_{\e_k},g_{\e_k}))_k$ found in Proposition \ref{P:4}. Then
 \begin{align}
& \int_{-L}^L\big[ f^2(T) + R (f+g)^2(T)\big]\zeta^2\, dx\nonumber\\
& \qquad + \int_0^T\int_{-L}^L \big(f \left| (1+R)\p_x f + R \p_xg \right|^2+ RR_\mu g  \left| \p_x f + \p_x g \right|^2\big)\zeta^2\, dx\, dt\nonumber \\
& \leq \int_{-L}^L\big[ f^2 (0) + R (f+g)^2(0)\big]\zeta^2\, dx\nonumber\\
& \qquad +4\int_0^T\int_{-L}^L \big[f \left( (1+R)f+Rg \right)^2+RR_\mu  g \left( f + g \right)^2\big] |\p_x\zeta|^2\, dx\, dt 
\label{REE}
\end{align}
for all $T>0$ and all $\zeta\in W^1_4(  (-L,L) )$.
\end{lemma}

\begin{proof} We set
\begin{align*}
& F_\e := \mathcal{R}_\e[f_\e]\ , \qquad G_\e := \mathcal{R}_\e[g_\e], \\
& U_\e := \sqrt{f_\e} \p_x \left[ (1+R) f_\e + R G_\e \right], \qquad V_\e := \sqrt{g_\e} \p_x \left[ F_\e + g_\e \right],
\end{align*}
and prove first the claim \eqref{REE} for $\zeta\in C^\infty_0((-L,L)).$ We multiply the first equation of \eqref{eq:A1} by $((1+R) f_\e + R G_\e)\zeta^2$ and integrate over $(-L,L)$ to obtain
\begin{align}
\int_{-L}^L \p_t f_\e\left( (1+R)f_\e+RG_\e \right)\zeta^2\, dx = & -\int_{-L}^L \sqrt{f_\e}  U_\e \p_x\left[\left( (1+R)f_\e+RG_\e \right)\zeta^2\right]\, dx + I_{1,\e} \label{eq:zz1}
\end{align}
with
$$
I_{1,\e}:=\e R\int_{-L}^L \p_x G_\e \p_x\left[\left( (1+R)f_\e+RG_\e \right)\zeta^2\right]\, dx.
$$
Similarly, multiplying the second equation of \eqref{eq:A1} by $R (F_\e+g_\e)\zeta^2$ and integrating over $(-L,L)$ give 
\begin{equation}
R\int_{-L}^L \p_t g_\e \left( F_\e + g_\e \right)\zeta^2\, dx = - RR_\mu \int_{-L}^L \sqrt{g_\e} V_\e \p_x\left[\left( F_\e + g_\e \right)\zeta^2\right]\, dx + I_{2,\e}
\label{eq:zz3}
\end{equation}
with 
$$
I_{2,\e}:=\e RR_\mu\int_{-L}^L\p_x F_\e \p_x\left[\left( F_\e + g_\e \right)\zeta^2\right]\, dx.
$$
We now observe that 
\begin{align}
& \int_{-L}^L \p_t f_\e\left( (1+R)f_\e+RG_\e \right)\zeta^2\, dx + R\int_{-L}^L \p_t g_\e \left( F_\e + g_\e \right)\zeta^2\, dx \nonumber \\
& \qquad = \frac{1+R}{2} \frac{d}{dt} \|f_\e\zeta\|_2^2 + \frac{R}{2} \frac{d}{dt} \|g_\e\zeta\|_2^2 + R J_{\e},\label{eq:yy1}
\end{align}
with 
\begin{align*}
J_{\e}:=& \int_{-L}^L\left( G_\e \p_t f_\e + F_\e \p_t g_\e \right)\zeta^2\, dx\\
=&\frac{d}{dt}\int_{-L}^L\big(F_\e G_\e+\e^2\p_x F_\e\p_x G_\e\big)\zeta^2\, dx+2\e^2\int_{-L}^L\left( G_\e \p_x\p_t F_\e + F_\e\p_x\p_tG_\e\right)\zeta\p_x\zeta\, dx.
\end{align*}
Recalling that $\zeta\in C^\infty_0((-L,L)),$ we have
\begin{align*}
\int_{-L}^L\big(F_\e G_\e+\e^2\p_x F_\e\p_x G_\e\big)\zeta^2\, dx  =& \frac{1}{2}\int_{-L}^L\big(F_\e G_\e+\e^2\p_x F_\e\p_x G_\e\big)\zeta^2\, dx \\
&   +\frac{1}{2}\int_{-L}^L\big(F_\e G_\e+\e^2\p_x F_\e\p_x G_\e\big)\zeta^2\, dx\\
 =& \frac{1}{2}\int_{-L}^L\big(F_\e G_\e-\e^2\p_x^2 F_\e  G_\e\big)\zeta^2\, dx \\
&   +\frac{1}{2}\int_{-L}^L\big(F_\e G_\e-\e^2 F_\e\p_x^2 G_\e\big)\zeta^2\, dx\\
&  - \e^2 \int_{-L}^L G_\e \p_x F_\e \zeta \p_x\zeta\, dx - \e^2 \int_{-L}^L F_\e\p_x G_\e\zeta\p_x\zeta\, dx\\
=& \frac{1}{2} \int_{-L}^L\big(F_\e g_\e+G_\e f_\e\big)\zeta^2\, dx+\e^2 \int_{-L}^LF_\e G_\e\p_x(\zeta\p_x\zeta)\, dx,
\end{align*}
while
\begin{align*}
\int_{-L}^L\big(G_\e \p_x\p_tF_\e +F_\e\p_x\p_tG_\e\big)\zeta\p_x\zeta\, dx  =& -\int_{-L}^L\big( \p_tF_\e \p_xG_\e+\p_xF_\e\p_tG_\e\big)\zeta\p_x\zeta\, dx\\
&  -\int_{-L}^L\big( G_\e \p_t F_\e +F_\e \p_tG_\e\big)\p_x(\zeta\p_x\zeta)\, dx\\
 =&-\int_{-L}^L\big( \p_tF_\e \p_xG_\e+\p_xF_\e\p_tG_\e\big)\zeta\p_x\zeta\, dx \\
&   -\frac{d}{dt}\int_{-L}^L F_\e G_\e\p_x(\zeta\p_x\zeta)\, dx.
\end{align*}

We end up with the following formula for $J_\e$:
\begin{align*}
J_\e  = &\frac{1}{2}\frac{d}{dt}\int_{-L}^L\big(F_\e g_\e+G_\e f_\e\big)\zeta^2\, dx-\e^2\frac{d}{dt}\int_{-L}^L F_\e G_\e\p_x(\zeta\p_x\zeta)\, dx\\
&  -2\e^2\int_{-L}^L\big( \p_tF_\e \p_xG_\e+\p_xF_\e\p_tG_\e\big)\zeta\p_x\zeta\, dx.
\end{align*}
After integration over $(0,T),$ it follows from \eqref{eq:zz1}-\eqref{eq:yy1} and the previous identity that
\begin{align*}
&\hspace{-0.3cm}\frac{1+R}{2} \|f_\e(T)\zeta\|_2^2-\frac{1+R}{2} \|f_\e(0)\zeta\|_2^2 + \frac{R}{2}  \|g_\e(T)\zeta\|_2^2-\frac{R}{2}  \|g_\e(0)\zeta\|_2^2 \\
&+\frac{R}{2} \int_{-L}^L\big(F_\e g_\e+G_\e f_\e\big)(T)\zeta^2\, dx- \frac{R}{2} \int_{-L}^L\big(F_\e g_\e+G_\e f_\e\big)(0)\zeta^2\, dx \\
&-R\e^2 \int_{-L}^L (F_\e G_\e)(T)\p_x(\zeta\p_x\zeta)\, dx+R\e^2 \int_{-L}^L (F_\e G_\e)(0)\p_x(\zeta\p_x\zeta)\, dx\\
&-2R\e^2\int_0^T\int_{-L}^L\big( \p_tF_\e \p_xG_\e+\p_xF_\e\p_tG_\e\big)\zeta\p_x\zeta\, dx\, ds\\
=&-\int_0^T\int_{-L}^L \sqrt{f_\e} U_\e \p_x\left[\left( (1+R)f_\e+RG_\e \right)\zeta^2\right]\, dx\, ds+\int_0^T I_{1,\e}\, ds\\
&- RR_\mu \int_0^T\int_{-L}^L \sqrt{g_\e} V_\e \p_x\left[\left( F_\e + g_\e \right)\zeta^2\right]\, dx\, ds + \int_0^TI_{2,\e}\, ds.
\end{align*}
Using Young's inequality we get
\begin{align}
&\hspace{-0.3cm}\frac{1+R}{2} \|f_\e(T)\zeta\|_2^2 + \frac{R}{2}  \|g_\e(T)\zeta\|_2^2 +\frac{R}{2} \int_{-L}^L\big(F_\e g_\e+G_\e f_\e\big)(T)\zeta^2\, dx\nonumber\\
&+K_\e(T) +\frac{1}{2}\int_0^T\int_{-L}^L \left[ U_\e^2 + RR_\mu V_\e^2 \right] \zeta^2\, dx\, ds\nonumber \\
\leq &\frac{1+R}{2} \|f_\e(0)\zeta\|_2^2 +\frac{R}{2}  \|g_\e(0)\zeta\|_2^2 + \frac{R}{2} \int_{-L}^L\big(F_\e g_\e+G_\e f_\e\big)(0)\zeta^2\, dx \nonumber\\
&+2\int_0^T\int_{-L}^L \big[f_\e  \left| (1+R)f_\e+RG_\e \right|^2 +RR_\mu g_\e \left| F_\e + g_\e \right|^2\big]|\p_x\zeta|^2\, dx\, ds \label{Mozart}
\end{align}
with
\begin{align*}
 K_\e(T):=&-R\e^2 \int_{-L}^L (F_\e G_\e)(T)\p_x(\zeta\p_x\zeta)\, dx+R\e^2 \int_{-L}^L (F_\e G_\e)(0)\p_x(\zeta\p_x\zeta)\, dx\\
&-2R\e^2\int_0^T\int_{-L}^L\big( \p_tF_\e \p_xG_\e+\p_xF_\e\p_tG_\e\big)\zeta\p_x\zeta\, dx\, ds
 - \int_0^T \left( I_{1,\e} + I_{2,\e} \right)\, ds.
\end{align*}

According to \cite{ELM11}, the convergences of $(f_{\e_k})_k$ and $(g_{\e_k})_k$ towards $f$ and $g$ actually take place in stronger topologies than stated in Proposition \ref{P4}. In fact,   for all $T>0$ 
 \begin{align}
 & f_{\e_k}\to f,\quad  F_{\e_k} \to f, \quad  g_{\e_k}\to g,\quad   G_{\e_k}\to g\qquad \text{in $L_2(0,T; C([-L,L])),$}\label{P1}\\
 & f_{\e_k}(0)\to f_0,\quad  F_{\e_k}(0) \to f_0, \quad  g_{\e_k}(0)\to g_0,\quad   G_{\e_k}(0)\to g_0\qquad \text{in $L_2((-L,L)),$}\label{P2}\\
 & U_{\e_k} \rightharpoonup U := \sqrt{f}\left((1+R)\p_xf+R\p_xg\right)\qquad \text{in $L_2((0,T)\times(-L,L))$}\label{P3}\\
 & V_{\e_k} \rightharpoonup V:= \sqrt{g}(\p_xf+\p_xg)\qquad \text{in $L_2((0,T)\times(-L,L)).$}\label{P4}
 \end{align}
Furthermore it follows from \cite[Lemmas~2.3 \&~2.5]{ELM11} that
\begin{equation}
((f_{\e_k} , g_{\e_k},F_{\e_k}, G_{\e_k}))_{k} \;\;\text{ are bounded in }\;\; L_\infty(0,T;L_2(-L,L))
 \cap L_2(0,T;H^1(-L,L)). \label{P5}
\end{equation}
We also infer from \eqref{P1} that
\begin{equation}
\lim_{k\to\infty} \big\{ \|(f_{\e_k}-f)(T)\|_2 + \|(F_{\e_k}-f)(T)\|_2 + \|(g_{\e_k}-g)(T)\|_2 + \|(G_{\e_k}-g)(T)\|_2 
\big\} = 0 \label{P6}
\end{equation}
for almost all $T>0$. 
We may then take $\e=\e_k$ in \eqref{Mozart} and let $k\to\infty$ to deduce from \eqref{P1}-\eqref{P6} that, for almost all $T>0$,
\begin{align}
&\hspace{-0.3cm}\frac{1+R}{2} \|f(T)\zeta\|_2^2 + \frac{R}{2}  \|g(T)\zeta\|_2^2 + R \int_{-L}^L (fg)(T)\zeta^2\, dx \nonumber\\
& + \frac{1}{2}\int_0^T\int_{-L}^L \left[ U^2 + RR_\mu V^2 \right] \zeta^2\, dx\, ds\nonumber \\
\leq& \frac{1+R}{2} \|f_0\zeta\|_2^2 +\frac{R}{2}  \|g_0\zeta\|_2^2 + R \int_{-L}^L f_0 g_0 \zeta^2\, dx \nonumber\\
&+2\int_0^T\int_{-L}^L \big[f  \left| (1+R)f + R g \right|^2 +RR_\mu g \left| f + g \right|^2\big]|\p_x\zeta|^2\, dx\, ds \label{Mozart2}
\end{align}
provided we establish that
\begin{equation}\label{Mo1}
\lim_{k\to \infty} K_{\e_k}(T) = 0 .
\end{equation}

\paragraph{\textbf{The term $K_\e(T)$}}

We are left with proving \eqref{Mo1} and actually identifying the behavior of $K_\e(T)$ as $\e\to 0$. Owing to \eqref{P2} and \eqref{P6}, it is clear that 
\begin{equation}
\lim_{k\to \infty} R \e_k^2 \int_{-L}^L (F_{\e_k} G_{\e_k})(T)\p_x(\zeta\p_x\zeta)\, dx = \lim_{k\to\infty} R\e_k^2 \int_{-L}^L (F_{\e_k} G_{\e_k})(0)\p_x(\zeta\p_x\zeta)\, dx = 0 \label{P7}
\end{equation}
for almost all $T>0$. It next readily follows from \eqref{P5} that 
\begin{equation}
\lim_{\e\to 0} \int_0^T I_{1,\e}\, dt = \lim_{\e\to 0} \int_0^T I_{2,\e}\, dt = 0 . \label{P8}
\end{equation}
Finally, since 
\[
\p_t f_\e = \partial_x \big( \sqrt{f_\e} U_\e - R \e \p_x G_\e \big),
\]
the boundedness \eqref{P3} and \eqref{P5} of $(U_\e)_\e$ in $L_2((0,T)\times (-L,L))$ and $(f_\e)_\e$ in $L_\infty(0,T;L_2(-L,L))$ implies that  $\left( \sqrt{f_\e} U_\e - R \e \p_x G_\e \right)_\e$ is bounded in $L_2(0,T;L_4(-L,L)).$ Consequently, the sequence
 $\left( \p_t f_\e \right)_\e$ is bounded in $L_2(0,T;(W_4^1(-L,L))')$ and so is $\left( \p_t g_\e \right)_\e$ by a similar argument. Owing to the properties of $\left( 1 - \e^2 \p_x^2 \right)^{-1}$ we conclude that 
\begin{equation}\label{Mo11}
\text{ $\left( \p_t F_\e \right)_\e, \left( \p_t G_\e\right)_\e$ are bounded in $L_2(0,T;(W_4^1)')$,}
\end{equation}
see \cite[Lemma~3.1]{ELM11} for a similar result. Now, since $\zeta\in C_0^\infty((-L,L))$ and $W_4^1(-L,L)$ is an algebra, we infer from \eqref{Mo11} that 
\begin{align}
&2R\e^2 \left| \int_0^T\int_{-L}^L\big( \p_tF_\e \p_xG_\e+\p_xF_\e\p_tG_\e\big)\zeta\p_x\zeta\, dx\, ds\right|\nonumber \\
 &\leq C \e^2 \int_0^T \Big( \|\p_t F_\e\|_{(W_4^1)'}  \|\p_x G_\e \|_{W^1_4}
+ \|\p_t G_\e\|_{(W_4^1)'}  \|\p_x F_\e\|_{W^1_4} \Big) \|\zeta\p_x\zeta\|_{W^1_4} \, ds \nonumber \\
 &\leq C(\zeta,T) \e^2 \left[\int_0^T   \Big(\|\p_xG_\e \|_{W^1_4}^2+    \|\p_xF_\e \|_{W^1_4}^2 \Big) \, ds\right]^{1/2}   . \label{P9}
\end{align}
Now, owing to \eqref{spirou}, for almost all $t\in(0,T)$ the function $\p_x F_\e(t)$ solves
$$
\p_x F_\e - \e^2 \p_x^2 \p_x F_\e = \p_x f_\e \;\;\text{ in }\;\; (-L,L)  , \qquad \p_x F_\e(\pm L) =0  ,
$$
which implies that
$$
\|\p_x F_\e\|_2^2 + \e^2 \|\p_x^2 F_\e\|_2^2 + \e^4\|\p_x^3 F_\e\|_2^2 \le C \|\p_x f_\e\|_2^2  .
$$
These estimates along with the Gagliardo-Nirenberg inequality \cite[Theorem 1]{N66} give
\begin{align*}
\|\p_x F_\e\|_{W_4^1} & \le C \left( \|\p_x F_\e\|_4 + \|\p_x^2 F_\e\|_4 \right) \\
& \le C \Big( \|\p_x^2 F_\e\|_2^{1/4} \|\p_x F_\e\|_2^{3/4} + \|\p_x^3 F_\e\|_2^{1/4} \|\p_x^2 F_\e\|_2^{3/4} + \|\p_x^2 F_\e\|_2 \Big) \\
& \le C \e^{-5/4} \|\p_x f_\e \|_2\ .
\end{align*}
A similar estimate being valid for $\|\p_x G_\e\|_{W_4^1}$ with $\|\p_x g_\e \|_2$ instead of $\|\p_x f_\e \|_2$, we deduce from \eqref{P5} and \eqref{P9} that
\begin{align}
& \hspace{-0.3cm}2R\e^2 \left| \int_0^T\int_{-L}^L\big( \p_tF_\e \p_xG_\e+\p_x F_\e\p_tG_\e\big)\zeta\p_x\zeta\, dx\, ds\right| \nonumber \\
\le & C(\zeta,T) \e^{3/4} \left[ \int_0^T \left( \|\p_x f_\e\|_2^2 + \| \p_x g_\e \|_2^2 \right)\, ds \right]^{1/2} \le C(\zeta,T) \e^{3/4}  . \label{P10}
\end{align}
Combining \eqref{P7}, \eqref{P8}, and \eqref{P10} gives the claim \eqref{Mo1} and completes the proof of \eqref{REE} for $\zeta\in C^\infty_0(\R),$ its validity for all $T>0$ being obtained by a lower semicontinuity argument. According to the regularity of $(f,g)$ the extension of Lemma~\ref{L:4} to all functions $\zeta\in W_4^1((-L,L))$ follows by a density argument.
 \end{proof}\medskip
 
 \subsection{The limit $L\to\infty$} For each positive $L$, we denote  the couple found in Proposition \ref{P:4} by $(f^L,g^L)$.
The family $((f^L,g^L))_L$ satisfies the same bounds as the family $((f_\e,g_\e))_\e$, so that performing the limit $L\to\infty$
may be done as the limit $\e\to0$, the only difference being the unboundedness of the domain which one has to cope with.
To this end we derive the following lemma which controls the behavior at infinity of  $(f^L,g^L)$. 
 \begin{lemma} \label{L36}It holds that
\begin{align}\label{M2}
\int_{-L/2}^{L/2}\Big(f^L+\frac{R}{R_\mu}g^L\Big)(T)x^2\, dx
\leq &  \int_{-L}^L\Big(f_0+\frac{R}{R_\mu}g_0\Big)x^2\, dx
+ T\E(f_0,g_0)
\end{align}
for all $T>0.$
 \end{lemma}
 \begin{proof}
 We define the function
 \[
\Phi(x)=
\left\{
\begin{array}{cll}
-2Lx-x^2-3L^2/4&,& -L \leq x\leq- L/2,\\
x^2&,& -L/2\leq x\leq L/2,\\
2Lx-x^2-3L^2/4&,& L/2\leq x\leq L.
\end{array}
\right.
\]
We take $\xi=\Phi$ in \eqref{FEQ} and $\xi=R\Phi/R_\mu$ in \eqref{GEQ} to obtain that,  using integration by parts
and the bound $\Phi''\le2$,
\begin{align*}
\int_{-L}^L\Big(f^L +\frac{R}{R_\mu}g^L\Big)(T)\Phi\, dx= &\int_{-L}^L\Big(f_0+\frac{R}{R_\mu}g_0\Big)\Phi\, dx
+\frac{1}{2}\int_0^T\int_{-L}^L\big((f^L)^2+ R(f^L+g^L)^2\big)\Phi''\, dx\, dt\\
\leq &  \int_{-L}^L\Big(f_0+\frac{R}{R_\mu}g_0\Big)\Phi\, dx
+ T\E(f_0,g_0).
\end{align*}
In addition 
\[
x^2{\bf 1}_{[-L/2,L/2]}\leq \Phi(x)\leq x^2,\qquad \text{for $x\in[-L,L],$}
\]
and the claim follows.
\end{proof}

Thanks to  Lemma \ref{L36} we may argue as   in the  proof of Proposition \ref{P4}, see \cite{ELM11},
to perform the limit $L\to\infty$ and complete the proof of Theorem \ref{T:1}.
We in particular use Lemma \ref{L36} to establish the entropy inequality $(b)$ as well as the conservation of mass in Theorem \ref{T:1}.
\bibliographystyle{abbrv}
\bibliography{LM_FSP} 

\end{document}